\documentclass[leqno,12pt]{amsart} 
\usepackage{amssymb}
\usepackage{amsmath}
\usepackage{amsthm}
\usepackage{amscd}
\usepackage{graphicx}
\usepackage[dvips]{color}
\usepackage[all]{xy}
\usepackage{tikz-cd}

\usepackage{url}
\usepackage{comment}
\theoremstyle{plain} 
\newtheorem{theorem}{\indent\sc Theorem}[section]
\newtheorem{lemma}[theorem]{\indent\sc Lemma}
\newtheorem{corollary}[theorem]{\indent\sc Corollary}
\newtheorem{proposition}[theorem]{\indent\sc Proposition}

\theoremstyle{definition} 
\newtheorem{definition}[theorem]{\indent\sc Definition}
\newtheorem{remark}[theorem]{\indent\sc Remark}
\newtheorem{example}[theorem]{\indent\sc Example}

\newcommand{\ddbar}{\partial \bar{\partial}}
\newcommand{\dbar}{\bar{\partial}}

\begin{document}

\title[Singular Hermitian metrics on vector bundles]{Approximations and examples of singular Hermitian metrics on vector bundles} 

\author[G. Hosono]{Genki Hosono} 

%
\address{
Graduate School of Mathematical Sciences, The University of Tokyo \endgraf
3-8-1 Komaba, Meguro-ku, Tokyo, 153-8914 \endgraf
Japan
}
\email{genkih@ms.u-tokyo.ac.jp}

\maketitle

\begin{abstract}
We study singular Hermitian metrics on vector bundles.
There are two main results in this paper.
The first one is on the coherence of the higher rank analogue of multiplier ideals for singular Hermitian metrics defined by global sections.
As an application, we show the coherence of the multiplier ideal of some positively curved singular Hermitian metrics whose standard approximations are not Nakano semipositive.
The aim of the second main result is to determine all negatively curved singular Hermitian metrics on certain type of vector bundles, for example, certain rank 2 bundles on elliptic curves.
\end{abstract}


\section{Introduction}
The main purpose of this paper is to investigate properties of {\it singular Hermitian metrics} on vector bundles on complex manifolds.
In complex algebraic geometry, singular Hermitian metrics on line bundles and their {\it multiplier ideal sheaves} are very important and widely used.
The higher rank analogue of these notions in the vector bundle case are also considered and investigated in many papers (for example, \cite{BP}, \cite{deC}, \cite{Lem}, \cite{PT}, \cite{Rau}, \cite{Y}, etc.).
There are several nonequivalent definitions for singular Hermitian metrics on vector bundles. 
Here we adopt the following most general definitions (\cite{BP}, \cite{Rau}, see Definition \ref{shmdef}, \ref{Grifpos}). 
Let $X$ be a complex manifold and let $E$ be a holomorphic vector bundle on $X$.
A {\it singular Hermitian metric} on $E$ is a measurable function on $X$ whose values are nonnegative Hermitian forms.
We say that $h$ is {\it negatively curved} if $|s|_h^2$ is plurisubharmonic for every local holomorphic section $s \in \mathcal{O}(E)$, i.e.\ for every holomorphic section of $E$ on each open subset of $X$. We say that $h$ is {\it positively curved} if the dual metric $h^*$ is well-defined and negatively curved.
In general, appropriate definition of curvature currents is not known for general singular Hermitian metrics on vector bundles. Thus we avoid using a curvature current of a singular Hermitian metric to define the positivity and the negativity here.
Instead we use a characterization of Griffiths seminegativity of smooth metrics (Lemma \ref{CharGrifNeg}).

We have two main results in this paper. 
The first one is on the coherence of a higher rank analogue of multiplier ideal sheaves.
Let $h$ be a singular Hermitian metric on $E$.
Following \cite{deC} and \cite{Lem}, we denote by $E(h)$ {\it the sheaf of locally square integrable holomorphic sections of $E$ with respect to $h$}, i.e.\ $E(h)$ is a sheaf of local holomorphic sections $s$ of $E$ which satisfy $|s|^2_h \in L^1_{loc}$.
The coherence of these sheaves is a basic problem, and the line bundle case was established by Nadel \cite{Nad}.
The first main result is as follows.

\begin{theorem}\label{MainTheorem1}
Let $X$ be a complex manifold, $E$ be a holomorphic vector bundle on $X$, and $s_1, s_2,\ldots,s_N \in H^0(X,E)$ be a global holomorphic section of $E$.
Assume that there exists a Zariski open set $U$ such that, for each $x \in U$, a fiber $E_x$ over $x$ is generated by $\{s_i(x)\}_{i=1,2,\ldots,N}$.
Define a morphism of vector bundles
\[\phi: X \times \mathbb{C}^N \to E\]
by sending $(x,(a_1,a_2,\ldots,a_N))$ to $\sum a_i s_i(x)$.
Let $h_0$ be a smooth Hermitian metric on $X \times \mathbb{C}^N$ and let $h$ be the quotient metric of $h_0$ induced by $\phi$.
Then, the sheaf $E(h)$ of locally square integrable sections of $E$ with respect to $h$ is a coherent subsheaf of $\mathcal{O}(E)$.
\end{theorem}

Here we note that $\phi$ is surjective on a Zariski open set U of $X$ by assumption, thus $h$ is well-defined pointwise on U. Therefore $h$ is well-defined as a singular Hermitian metric (see Section 3 for well-definedness of singular Hermitian metrics).

Thus we can solve the problem of coherence of $E(h)$ for metrics of this particular form. On the other hand, if locally there exists a sequence of smooth Hermitian metrics $\{h_j\}$ satisfying
$h_j \uparrow h$ pointwise and $\Theta_{h_j} >_{\rm Nak} \gamma \otimes {\rm Id}_E$ for a fixed continuous (1,1)-form  $\gamma$,
we have that $E(h)$ is coherent (\cite[Section 4]{deC}).
For proving this, we use H\"ormander's $L^2$ estimate. 
The condition that Chern curvature is bounded below in the sense of Nakano is important (although they can be weakened) for using $L^2$ theory.
For a more detailed discussion, see \cite[Section 3]{Rau}.
We construct an example of a singular Hermitian metric which has the form of Theorem \ref{MainTheorem1}, but its standard approximation does not have uniformly bounded curvature in the sense of Nakano (cf. Example \ref{exampleapproximate}).

\begin{theorem}\label{approximationex}
Let $E = X \times \mathbb{C}^2$ be a trivial vector bundle of rank two on $X=\mathbb{C}^2$.
Let $h_0$ be the standard Euclidean metric on $E$.
We denote by $(z,w)$ the standard coordinate on $X$ and let $s_1 =(1,0), s_2=(z,w)$ be global sections of $E$.
We define a singular Hermitian metric $h$ on $E$ as in Theorem \ref{MainTheorem1} using $s_1, s_2$.
Then, the standard approximation defined by convolution of $h$ does not have uniformly bounded curvature from below in the sense of Nakano.
\end{theorem}

As a corollary of Theorem \ref{MainTheorem1}, we prove the coherence of $E(h)$ even in this situation.

\begin{corollary}\label{exampleofmainthm1}
For $h$ constructed in Theorem \ref{approximationex}, we have that $E(h)$ is coherent.
\end{corollary}

Thus Theorem \ref{MainTheorem1} does not seem to be a consequence of the results in \cite{deC}.

Our second main theorem is on the determination of singular Hermitian metrics on certain vector bundles.
We are mainly interested in determining all singular Hermitian metrics with certain curvature-positivity conditions on the vector bundle in \cite[Example 1.7]{DPS}. 
First we explain this example. Let $C$ be an elliptic curve and let $E$ be the vector bundle defined by the following non-splitting exact sequence:
\[0 \longrightarrow \mathcal{O}_C \longrightarrow E \longrightarrow \mathcal{O}_C \longrightarrow 0.\]
It is known that the bundle $E$ satisfying this condition is unique.
In \cite[Example 1.7]{DPS}, a line bundle $L=\mathcal{O}_{\mathbb{P}(E)}(1)$ on the projectivization $\mathbb{P}(E)$ is considered. They determined all singular Hermitain metrics on $L$ with positive curvature.
In this paper, we determine all negatively curved singular Hermitian metrics on $E$.
This is a parallel result with \cite[Example 1.7]{DPS}, since singular Hermitian metrics on $L$ can be regarded as ``singular Finsler metrics'' on $E$.
Our result is the following:

\begin{theorem}\label{MainTheorem2}
Let $X$ be a compact complex manifold, $L$ be a holomorphic line bundle on $X$, and $E, E'$ be holomorphic vector bundles on $X$.
Assume there is an exact sequence
\[0 \longrightarrow L \overset{i}\longrightarrow E \overset{p}\longrightarrow E' \longrightarrow 0.\]
Suppose that there are a holomorphic section $f \in H^0(X,L^*)$ and a negatively curved singular Hermitian metric $h$ on $E$ with $|i(s)|^2_h = | (f,s) |^2$ for each $s \in L$, where $( \cdot, \cdot )$ is the natural pairing on $L_x^* \times L_x$.
Then, this exact sequence splits.
\end{theorem}

Our technique is different from \cite{DPS}.
To make the argument clear, we consider the simpler case that $h$ is a smooth Hermitian metric (in particular, $h$ is positive definite on every point of $X$).
In this case, $|i(s)|^2_h = | (f,s) |^2 \neq 0$ for non-zero $s \in L$, thus $f$ is non-vanishing. Therefore $L^*$ must be isomorphic to the trivial line bundle $\mathcal{O}_X$.
In addition, for simplicity, we make additional assumption that $E'$ is also a trivial line bundle.
We will sketch the proof in this case.

Let $\{U_\alpha\}$ be a covering of $X$ such that each $U_\alpha$ is enough small.
We can assume $E|_{U_\alpha}$ is isomorphic to a trivial vector bundle.
Take a holomorphic frame $(e_{\alpha,1}, e_{\alpha,2})$ of $E$ satisfying $i(1) = e_{\alpha,1}$ and $p(e_{\alpha, 2}) = 1$.
Then there is a holomorphic function $h_{\alpha\beta}$ on $U_{\alpha} \cap U_{\beta}$ with 
$e_{\alpha, 2} - e_{\beta,2} = h_{\alpha\beta} e_{\beta,1}$.
Let $\gamma_\alpha := \langle e_{\alpha,2},e_{\alpha,1}\rangle_h$, where $\langle \cdot, \cdot \rangle_h$ denotes the Hermitian inner product of $h$.
Then we have
\[\gamma_{\alpha} = h_{\alpha\beta} + \gamma_{\beta}.\]
By the negativity of $h$, we can show that $\gamma_{\alpha}$ is a holomorphic function on $U_\alpha$ (see Proposition \ref{locrep}).
Consider \v{C}ech 0-cochain $\left(\gamma_{\alpha}, U_\alpha\right)_\alpha$.
Its differential is the \v{C}ech 1-cocycle $(h_{\alpha\beta}, U_{\alpha\beta})_{\alpha,\beta}$, which is identical to the extension class of the given exact sequence.
Thus the extension class is 0 in $H^1(X,\mathcal{O}_X)$ and the given sequence is trivial.

The most important point is that negativity of $h$ implies holomorphicity of $\gamma_{\alpha}$. The proof in the general case is similar to this argument, but it requires more complicated calculation.

As a corollary of Theorem \ref{MainTheorem2}, we determine all negatively curved singular Hermitian metrics on $E$.

\begin{corollary}
Let $C$ be an elliptic curve, and $E$ be a rank 2 bundle defined by the non-splitting exact sequence
\[0 \longrightarrow \mathcal{O}_C \longrightarrow E \overset{p}\longrightarrow \mathcal{O}_C \longrightarrow 0.\]
Let $h$ be a negatively curved singular Hermitian metric on $E$.
Then $h = p^*h'$ holds for a constant metric $h'$ on $\mathcal{O}_C$, where $p^* h'$ is defined by the formula $|s(z)|_{p^*h'}:=|p(s(z))|_{h'}$ for any $z \in C$ and a section $s$ of $E$.
\end{corollary}

The organization of the paper is as follows.
In Section \ref{Preliminaries}, we collect preliminary materials related to smooth Hermitian metrics on vector bundles and singular Hermitian metrics on line bundles.
Section \ref{SingHermMetonVB} contains the definition of singular Hermitian metrics on vector bundles and their properties described in \cite{deC}, \cite{BP} and \cite{Rau}.
This section also contains some examples of singular Hermitian metrics on vector bundles.
In Section \ref{SingHermMetbyGlobalSections} and \ref{ClassifyGrifNegSingMet}, we prove the main theorems and make some discussions.

\subsection*{\indent Acknowledgments}
The author would like to thank his supervisor Prof.\ Shigeharu Takayama for enormous supports and insightful comments.
He would also like to thank Dr.\ Takayuki Koike for discussions and valuable comments. This work is supported by the Program for Leading Graduate Schools, MEXT, Japan.
This work is also supported by JSPS KAKENHI Grant Number 15J08115.

\section{Hermitian metrics on vector bundles and singular Hermitian metrics on line bundles}\label{Preliminaries}

In this section, we review some classical notions such as smooth Hermitian metrics on vector bundles and singular Hermitian metrics on line bundles.
Basic references for this section are \cite{DemCG} and \cite{DemAM}.

\subsection*{\indent Notation.}
Throughout this paper, $X$ denotes a complex manifold and $E$ denotes a holomorphic vector bundle on $X$.
For $x \in X$, $E_x$ denotes the fiber of $E$ over $x$.
We use the notation $s \in E$ for denoting a vector in some fiber of $E$, i.e.\ a point of the total space of $E$.
The sheaf of holomorphic sections of $E$ is denoted by $\mathcal{O}(E)$.
The natural pairing $E_x^* \times E_x \to \mathbb{C}$ on each fiber is denoted by $(\cdot,\cdot)$.
We denote by ${}^t A$ the transpose of a matrix $A$.

\subsection{Smooth Hermitian metrics on vector bundles}\label{SmHermMetVB}
Let $h$ be a Hermitian metric on $E$, i.e.\ $h$ defines a positive definite Hermitian inner product $\langle \cdot,\cdot \rangle_h$ on each fiber $E_x$ and, for any smooth local section $s,t$ of $E$, $\langle s,t \rangle_h$ is a smooth function.
We denote by $|\cdot|_h$ the norm on $E$.
If we fix a local holomorphic frame $(s_1, s_2, \ldots, s_r)$ of $E$, we can identify $s = \sum c_i s_i \in E$ with a column vector ${}^t(c_1, c_2, \ldots, c_r)$.
Then $h$ has a matrix representation defined by
\[\langle s,t \rangle_h = {}^t s h \overline{t}.\]

The {\it Chern curvature} $\Theta$ of $(E,h)$ is an ${\rm End}(E)$-valued (1,1)-form locally defined by $\Theta = \dbar(\overline{h}^{-1}\partial \overline{h})$. 
For $s, t \in E$ and differential forms $\alpha, \beta$, we introduce the following notation
\[\{ s \otimes \alpha, t \otimes \beta\}_h := \langle s, t \rangle_h \alpha \wedge \overline{\beta}.\] 

Here we recall the notion of the positivity.
\begin{definition}\label{DefSmPos}
Let $\Theta$ be a Hermitian form on $E \otimes T_X$.\\
(1) We say that $\Theta$ is {\it Griffiths semipositive} (resp.\ {\it Griffiths positive}) if $\Theta(s \otimes \xi) \geq 0$ (resp.\ $> 0)$ for every local section $s \in E, \xi \in T_X$.\\
(2) We say that $\Theta$ is {\it Nakano semipositive} (resp.\ {\it Nakano positive}) if $\Theta(\sum s_i \otimes \xi_i) \geq 0$ (resp.\ $> 0)$ for every local section $s_i \in E, \xi_i \in T_X, i=1,2,...,{\rm min}(n,r)$.

We say that $(E,h)$ is {\it Griffiths semipositive} if the Hermitian form $\Theta(s\otimes\xi, t\otimes\eta) := \{\Theta s,t\}_h(\xi,\overline{\eta})$ on $E \otimes T_X$ defined by its Chern curvature $\Theta$ is Griffiths semipositive.
It is equivalent to the condition that $i\{\Theta s, s\}_h$ is a positive $(1,1)$-form for every $s \in E$.
When a Hermitian form $\Theta$ on $E \otimes T_X$ is Griffiths  (resp.\ Nakano ) semipositive, we write $\Theta \geq_{\rm Grif} 0$ (resp.\ $\geq_{\rm Nak} 0$).
For two Hermitian forms $\Theta_1$ and $\Theta_2$, we write $\Theta_1 \geq_{\rm Nak} \Theta_2$ when $\Theta_1 - \Theta_2 \geq_{\rm Nak} 0$.
\end{definition}

If ${\rm dim} \, X = 1$ or ${\rm rank} \, E = 1$, Griffiths positivity and the Nakano positivity are equivalent.
Griffiths positivity has nice functorial properties.
For example, it is preserved under the quotient and the dual of a Griffiths positive vector bundle is Griffiths negative.
These properties do not hold for Nakano positivity.
Nakano positivity is often used to describe a condition for applying $L^2$-methods.

We give the following characterization of Griffiths negativity. This will be used as a definition in the singular case.

\begin{lemma}[{\cite[Section 2]{Rau}}]\label{CharGrifNeg}
Let $h$ be a smooth Hermitian metric on $E$. Then, the following conditions are equivalent.\\
(1) $h$ is Griffiths seminegative (as in Definition \ref{DefSmPos}).\\
(2) For every local holomorphic section $s \in \mathcal{O}(E)$, $|s|_h^2 = \langle s,s\rangle_h$ is plurisubharmonic.\\
(3) For every local holomorphic section $s \in \mathcal{O}(E)$, $\log |s|_h^2$ is plurisubharmonic.
\end{lemma}

\subsection{Singular Hermitian metrics on line bundles and plurisubharmonic functions}

On a line bundle, a singular Hermitian metric is also important. We begin with the definitions.
\begin{definition}[cf.\ \cite{DemAM}]\label{def:shmlb}
Let $L$ be a holomorphic line bundle on $X$. A {\it singular Hermitian metric} $h$ on $L$ is a measurable metric on $L$ with locally integrable weight function, i.e.\ locally $h$ has the form $|\cdot|^2_h =|\cdot|^2 e^{-\phi}$ for some $L^1_{loc}$ function $\phi$.
\end{definition}

We note that two singular Hermitian metrics $h, h'$ are equal when $\phi=\phi'$ holds for their local weight.
In this definition, the assumption that $\phi$ is locally integrable ensures the existence of the curvature current.
Formally we have $\Theta = -\partial \dbar \log h = \partial \dbar \phi$. Since $\phi$ is locally integrable, the right-hand side can be defined in the sense of currents.
We can also show that the right-hand side is independent of the choice of trivialization, thus the curvature current is globally well-defined.

A singular Hermitian metric $h$ on a line bundle $L$ is said to be {\it positively curved} if its curvature current $\Theta_h$ satisfies $\Theta_h \geq 0$ in the sense of currents.
Note that $h$ is positively curved if and only if its local weight $\phi$ is plurisubharmonic.

For a plurisubharmonic function $\phi$, the associated {\it multiplier ideal sheaf} $\mathcal{I}(\phi)$ is defined by
\[\mathcal{I}(\phi)(U) = \{f \in \mathcal{O}(U) ; |f|^2 e^{-\phi} \in L^1_{loc} (U) \}.\]
For a positively curved singular Hermitian metric $h = e^{-\phi}$, $\mathcal{I}(\phi)$ is independent of a choice of local trivialization.
Therefore $\mathcal{I}(h) := \mathcal{I}(\phi)$ is well-defined.
The coherence of $\mathcal{I}(h)$ is obtained by Nadel using $L^2$-methods:

\begin{theorem}[{\cite{Nad}, \cite[Proposition 5.7]{DemAM}}]\label{MultIdealIsCoh}
For any plurisubharmonic function $\phi$, $\mathcal{I}(\phi)$ is coherent. It follows that, for any positively curved singular Hermitian metric $h$, $\mathcal{I}(h)$ is coherent.
\end{theorem}

This theorem is obtained by using $L^2$-methods. The higher rank analogue of this theorem also holds under the assumption on the Nakano curvature condition which ensures that we can use $L^2$-methods (see Proposition \ref{cohvb}).

\section{Definition and examples of Singular Hermitian metrics on vector bundles}\label{SingHermMetonVB}
In this section, we introduce the notion of a singular Hermitian metric on a holomorphic vector bundle.
By de Cataldo, metrics approximated by smooth metrics in $C^2$-topology on an open set are considered \cite{deC}.
This approach is suitable to consider the Nakano positivity condition to use $L^2$-methods.
More general concept is given by Berndtsson and Paun in \cite{BP}, where all measurable metrics are considered.
We follow this approach and consider curvature conditions similar to \cite{deC} when applying $L^2$-methods.

\begin{definition}[{\cite[Section 3]{BP} and \cite[Section 1, Definition 1]{Rau}}]\label{shmdef}
Let $X$ be a complex manifold, and $E$ be a holomorphic vector bundle on $X$.
A {\it singular Hermitian metric} $h$ on $E$ is a collection of nonnegative Hermitian forms $h_x$ on $E_x$ for almost every $x \in X$ such that, for every holomorphic section $s,t$ of $E$, $\langle s,t \rangle_h$ is a measurable function.
In this paper, we admit the case $\det h \equiv 0$.

We also admit the case $h(s_x) = \infty$ for some $s_x \in E_x$. In this case, we assume that the set $\{x\in X ; h(s_x) = \infty \text{ for some } s_x \in E_x\}$ has zero measure. 
Here, a Hermitian metric on a complex vector space $V$ with values in $[0,\infty]$ is defined as follows:
there exists a subspace $V_0 \subset V$ which satisfies that $h|_{V_0}$ is a nonnegative ordinary Hermitian metric and $h(s) = \infty$ for all $s \in V \setminus V_0$.

Two singular Hermitian metrics $h, h'$ are said to be equal if $h=h'$ a.e.
\end{definition}

Because this definition is too general to deal with, we should restrict our interest to singular Hermitian metrics with appropriate curvature condition.
However, it seems difficult to define Chern curvature forms or currents for general singular Hermitian metrics (see Section 3 in \cite{Rau}).
Thus, alternatively, we use the characterization described in Lemma \ref{CharGrifNeg} to define the curvature condition of a singular Hermitian metric in the sense of Griffiths.

\begin{definition}[{\cite[Definition 3.1]{BP} and \cite[Section 1, Definition 2]{Rau}}]\label{Grifpos}
Let $h$ be a singular Hermitian metric on a holomorphic vector bundle $E$.\\
(1) $h$ is {\it negatively curved} (or {\it Griffiths seminegative}) if $|s|_h^2$ is plurisubharmonic for every local holomorphic section $s$ of $E$.\\
(2) $h$ is {\it positively curved} (or {\it Griffiths semipositive}) if the dual metric $h^*$ is well-defined and negatively curved.
\end{definition}

Note that, in our definition, the dual metric of $h$ is well-defined as a singular metric if ${\rm det} \ h \neq 0$ a.e.
Therefore we define the notion of a positively curved metric only for singular metrics with ${\rm det} \ h \neq 0$ a.e.
We can also define a Nakano negative singular Hermitian metric (cf.\ \cite{Rau}) although we will not use the Nakano negativity in this paper.
If $L$ is a line bundle with a singular Hermitian metric $h$, the positivity conditios in Definition \ref{Grifpos} and after Definition \ref{def:shmlb} coinside.

As a higher rank analogue of multiplier ideal sheaves, we define a subsheaf $E(h)$ of $\mathcal{O}(E)$ as follows.

\begin{definition}[\cite{deC}, Definition 2.3.1]\label{MultSubSheaf}
For a singular Hermitian metric $h$ on $E$, $E(h)$ denotes the sheaf of locally square integrable holomorphic sections of $E$ with respect to $h$, i.e.\ holomorphic sections $s \in \mathcal{O}(E)$ such that $|s|^2_h \in L^1_{loc}$.
\end{definition}

If $L$ is a line bundle, $L(h) = L \otimes \mathcal{I}(h)$.
In the line bundle case, $\mathcal{I}(h)$ is coherent for positively curved $h$ (Theorem \ref{MultIdealIsCoh}).
In the vector bundle case, we also have that $E(h)$ is coherent under an assumption related to Nakano positivity.
The precise statement is as follows. 

\begin{proposition}[\cite{deC}, Proposition 4.1.3]\label{cohvb}
Let $h$ be a singular Hermitian metric on $E$.
Assume that locally there exists a sequence of smooth Hermitian metrics $\{h_j\}$ satisfying
$h_j \uparrow h$ pointwise and $\Theta_{h_j} >_{\rm Nak} \gamma \otimes {\rm Id}_E$, where $\gamma$ is a fixed continuous (1,1)-form.
Then we have that $E(h)$ is coherent.
\end{proposition}

This proposition also holds under the weaker assumption that the curvature of $h_j$ are uniformly bounded from below in the sense of Nakano, i.e.\ $h_j \geq_{\rm Nak} -C\omega$ for some constant $C$ and a Hermitian form $\omega$ on $X$.

To construct positively (or negatively) curved metrics, the following lemma is useful.

\begin{lemma}\label{PullbackofGrifNeg}
Let $E,F$ be vector bundles on $X$, and let $\phi: \mathcal{O}(E) \to \mathcal{O}(F)$ be a sheaf homomorphism.

$(1)$ Let $h$ be a negatively curved singular Hermitian metric on $F$.
Then, a singular Hermitian metric $\phi^* h$ on $E$ defined by
\[\langle s,t \rangle_{\phi^* h}=
\langle \phi(s),\phi(t) \rangle_{h}\]
is also negatively curved.

$(2)$ If $\phi: \mathcal{O}(E) \to \mathcal{O}(F)$ is surjective on an open set $U$ of full Lebesgue measure and $h$ is a singular Hermitian metric on $E$, then the quotient metric on $F$ is well-defined as a singular Hermitian metric.
Moreover, if $h$ is positively curved, so is the quotient metric.
\end{lemma}

\begin{proof}
$(1)$ For a holomorphic section $s$ of $E$, we have that $\phi(s)$ is a holomorphic section of $F$. Since $h$ is negatively curved, $|s|_{\phi^* h} = |\phi(s)|_h$ is plurisubharmonic.

$(2)$ We have that the quotient (singular) metric is well-defined on $F|_U$.
Since $U$ is an open set with full measure, we have a singular Hermitian metric on $F$.
If $h$ is positively curved, we have that the dual metric $h^*$ is a negatively curved singular Hermitian metric on $E^*$.
The quotient metric coincides with the dual of $(\phi^*)^*h^*$, where $\phi^* : \mathcal{O}(F^*) \to \mathcal{O}(E^*)$ is the dual of $\phi$.
It follows that the quotient metric is positively curved by (1).
\end{proof}

\begin{example}\label{inducedbysections}
We can construct a positively curved singular Hermitian metric using given global sections of a vector bundle.
Let $s_1, s_2, \ldots, s_N \in H^0(X,E)$ be global sections of a vector bundle $E$ on a complex manifold $X$.
Then we have the following morphism of bundles
\[\phi: X \times \mathbb{C}^N \to E\]
sending $(x,(a_1,a_2,\ldots,a_N))$ to $\sum a_i s_i(x)$.
We assume that there exists a Zariski open set $U$ on which these $\{s_i\}$ generate each fiber of $E$ (we refer to this condition as that these $\{s_i\}$ {\it generically generate} $E$).
Then that $\phi$ is surjective on $U$.
Thus, by Lemma \ref{PullbackofGrifNeg}, the quotient metric $h$ of the standard metric on $\mathbb{C}^N$ is a positively curved singular Hermitian metric on $E$.

In Proposition \ref{coh}, we prove that $E(h)$ is coherent.
Note that we do not know that $h$ can be approximated by smooth metrics with Nakano positive curvature (or metrics which have Nakano curvature bounded from below).
In Example \ref{exampleapproximate}, we will show an example of this kind of metric which does not seem to satisfy Nakano curvature condition.
This type of singular Hermitian metrics is also investigated in \cite[Example 3.6]{Y}.
\end{example}

\begin{example}\label{ToricExample}
We shall show that there are positively curved singular Hermitian metrics on tangent bundles of toric varieties.

Let $X$ be a toric manifold. By definition, there is a inclusion $(\mathbb{C}^*)^n \subset X$ and action $(\mathbb{C}^*)^n \curvearrowright X$.
Considering the differentiation of the family of actions $(e^{i^\theta},1,\ldots,1)$ at $\theta = 0$, we have a holomorphic vector field on $X$.
Similarly, we have $n = \dim X$ vector fields which generates $TX$ on $(\mathbb{C}^*)^n$.
Therefore, by Example \ref{inducedbysections}, we can construct a positively curved singular Hermitian metric on $TX$
Here, we can also use toric Euler sequence (cf.\ \cite[Theorem 8.1.6]{CLS} to construct such metrics on $TX$.

In particular, we have a positively curved singular Hermitian metric on $T_X$ on the one-point blow-up $X = {\rm Bl}_P \mathbb{P}^2$ of $\mathbb{P}^2$ at $P \in \mathbb{P}^2$. Note that $X$ is isomorphic to a toric variety associated to a complete fan $\Sigma$ defined by four rays $\mathbb{R}_+ (0,1), \mathbb{R}_+ (1,1), \mathbb{R}_+ (1,0), \mathbb{R}_+ (-1,-1)$. This gives an example of a positively curved singular Hermitian metric on a vector bundle with no smooth Griffiths semipositive Hermitian metrics.
Indeed, the exceptional divisor $C$ on $X$ is a $(-1)$-curve. Then we have an exact sequence
\[0 \longrightarrow T_C \longrightarrow  T_X|_C \longrightarrow N_{C/X} \longrightarrow 0\]
on $C$.
If there is a smooth Hermitian metric on $T_X$ with Griffiths semipositive curvature, its restriction to $C$ is also semipositive. Then we have a semipositive metric on $N_{C/X} \cong \mathcal{O}(-1)$, which is a contradiction.
\end{example}

\begin{example}
Let $E$ be a holomorphic vector bundle on $X$.
We denote by $\mathbb{P}(E)$ the projective bundle of hyperplanes of $E$ and by $\mathcal{O}(1) = \mathcal{O}_{\mathbb{P}(E)}(1)$ the tautological line bundle on $\mathbb{P}(E)$. We denote $\mathcal{O}(k):=\mathcal{O}(1)^{\otimes k}$.
Let $h$ be a singular Hermitian metric on a line bundle $\mathcal{O}(1)$.
Assume that $h|_{\mathbb{P}(E)_x}$ is bounded for almost all $x \in X$.
Then we can construct a singular Hermitian metric $h'$ on $E \otimes \det E$ as follows.
First, we define a vector space $F_x$, $x \in X$, by
\[F_x = H^0\left(\mathbb{P}(E)_x, \left(\mathcal{O}(r+1)\otimes K_{\mathbb{P}(E)/X}\right) |_{\mathbb{P}(E)_x}\right).\]
Then the collection $\{F_x\}_{x \in X}$ forms a holomorphic vector bundle $F$.
By computation of transition functions, we have that $F$ is isomorphic to $E \otimes \det E$.
We define a singular Hermitian metric $h'$ on $F$ by
\[\langle s,t \rangle_{h'} = \int_{\mathbb{P}(E)_x} \{s,t\}_{h^{r+1}}\]
for $s, t \in F_x$.
Here, we use the notation $\{\cdot,\cdot\}$ defined in Section 2.1.
Recall that when $s\otimes\alpha$ is an $E$-valued $p$-form and $t\otimes\beta$ is an $E$-valued $q$-form, the product $\{s\otimes\alpha,t\otimes\beta\}_h$ is a (scalar-valued) $(p+q)$-form.
In this situation, We consider $s, t$ as $\mathcal{O}(r+1)$-valued $(r-1,0)$-forms on the fiber $\mathbb{P}(E)_x$, then $\{s,t\}_{h^{r+1}}$ is a $(r-1,r-1)$-form.
Therefore we can integrate $\{s,t\}_{h^{r+1}}$ on the fiber $\mathbb{P}(E)_x$.

When $h$ is smooth and semipositive, it is known that $h'$ is Nakano semipositive (cf.\ \cite{Ber}, \cite{LSY}).
We want to show that {\it if $h$ is a semipositive singular Hermitian metric, we can locally approximate $h'$ by Nakano semipositive smooth Hermitian metrics}.
Let $U \subset X$ be a small open set. We assume that $E$ is trivial on $U$, then we can write $E = U \times \mathbb{C}^r$ and $\mathbb{P}(E) = U \times \mathbb{P}^{r-1}$.
The line bundle $\mathcal{O}(1)$ admits a smooth positive Hermitian metric on $U \times \mathbb{P}^{r-1}$.
By the argument similar to the proof of \cite[Theorem 1]{BK}, we can construct an approximation of $h$ by smooth semipositive metrics $\{h_j\}$ on $V \times \mathbb{P}^{r-1}$, where $V$ is a relatively compact subset of $U$.
Each $h_j$ induces a smooth Hermitian metric $h'_j$ on $E \otimes \det E |_V$, which is known to be Nakano semipositive.
Thus we can construct a local approximation of $h'$ by Nakano semipositive smooth Hermitian metrics.
\end{example}

\section{Singular Hermitian metrics induced by global sections}\label{SingHermMetbyGlobalSections}

In this section, we study singular Hermitian metrics on vector bundles which are induced by holomorphic sections (Example \ref{inducedbysections}). 

\begin{proposition}\label{coh}
Let $X$ be a complex manifold, $E$ be a holomorphic vector bundle on $X$, and $s_1, s_2,\ldots,s_N \in H^0(X,E)$ be holomorphic sections.
Assume that there exists an open dense set $U$ such that, for every $x\in U$,  a fiber $E_x$ over $x$ is generated by these $\{s_i(x)\}_i$.
Define a morphism of vector bundles
\[\phi: X \times \mathbb{C}^N \to E\]
by sending $(x,(a_1,a_2,\ldots,a_N))$ to $\sum a_i s_i(x)$.
By assumption, $\phi$ is surjective on a Zariski open set.
Let $h_0$ be a smooth Hermitian metric on $\mathbb{C}^N$ and $h$ be the quotient metric of $h_0$ induced by $\phi$ (When $h_0$ is the standard Euclidean metric on $\mathbb{C}^N$, it is constructed in Example \ref{inducedbysections}).
Then, the sheaf $E(h)$ of locally square integrable sections of $E$ with respect to $h$ is a coherent subsheaf of $\mathcal{O}(E)$.
\end{proposition}

\begin{remark}\label{remNak}
By Proposition \ref{cohvb}, $E(h)$ is coherent if $h$ can be approximated by smooth metrics $\{h_j\}$ such that the curvature is uniformly bounded below in the sense of Nakano (i.e. there is a constant $C>0$ satisfies
\[\Theta_{h_j} \geq_{\rm Nak} -C \omega \otimes {\rm Id}_E,\]
where $\omega$ denotes a fixed Hermitian form on $X$). Here, some metrics defined in Example \ref{inducedbysections} seem hard to approximate in such a manner. See Example \ref{exampleapproximate} after the proof.
\end{remark}

\begin{proof}
First, we consider the case that $h_0$ is the standard Euclidean metric.
To prove the proposition, we calculate the value of $|s|^2_h$ explicitly, and use the result in the line bundle case.
Since the statement is local, it is sufficient to show the proposition on a small open set $U$. We assume that $E$ is trivialized on $U$ by a holomorphic frame $e_1, e_2, \ldots, e_r$ on $U$.
We regard each $s_i = \sum_j f_{i,j} e_j$ as a column vector ${}^t(f_{i,1}, f_{i,2}, \ldots, f_{i,r})$.
For $s = {}^t(f_1,f_2,\ldots,f_r)$, we will show the following equation.

\begin{lemma}
\[|s|^2_h =
\frac{
\sum_{1\leq i_1<i_2< \cdots <i_{r-1}\leq N} \left|\det
(s \, s_{i_1} \, s_{i_2} \cdots s_{i_{r-1}} )
\right|^2
}{
\sum_{1\leq j_1<j_2< \cdots <j_{r}\leq N} \left|\det
(s_{j_1} \, s_{j_2} \cdots s_{j_r} )
\right|^2}.
\]
\end{lemma}

\begin{proof}
We recall that $h$ is a quotient metric of the standard metric on $X \times \mathbb{C}^N$ induced by $\phi: X \times \mathbb{C}^N \to E$.
By considering the dual, $h$ can be regarded as the dual metric of the restriction of the standard metric via $\phi^*: E^* \to \mathbb{C}^N$. We denote the dual of $h$ by $h^*$.
We have that $\phi^*(\xi) = (\xi(s_1),\cdots,\xi(s_N))$ for $\xi \in E^*$.
Therefore the norm of $\xi$ with respect to $h^*$ is 
\[|\xi|^2_{h^*} = \sum_{i=1}^N |\xi(s_i)|^2,\]
thus the matrix representation of $h^*$ is as follows:
\[h^*_{jk} = \sum_{i=1}^N f_{i,j}\overline{f_{i,k}}.\]
Since the dual metric $h^*$ can be represented as ${}^th^{-1}$, we have that
\[h^{-1}_{jk} = \sum_{i=1}^N \overline{f_{i,j}}f_{i,k}.\]
Now we shall prove that $\det (h^{-1})$ is the denominator of the right hand side of the equation above.
We have that
\begin{align*}
\det h^{-1} &= \sum_{\sigma \in S_r} {\rm sgn} (\sigma) \prod_{j=1}^r h^{-1}_{j\sigma(j)}\\
&= \sum_{\sigma \in S_r} {\rm sgn} (\sigma) \prod_{j=1}^r\sum_{i_j=1}^N\overline{f_{i_j,j}}f_{i_j,\sigma(j)}.
\end{align*}
Extending the product, it follows that
\begin{align*}
&= \sum_{i_1,\ldots,i_r=1}^{N} \sum_{\sigma \in S_r}{\rm sgn}(\sigma) \overline{f_{i_1,1}}f_{i_1,\sigma(1)}\cdots \overline{f_{i_r,r}}f_{i_r,\sigma(r)}.
\end{align*}
When some two of $i_1,\ldots,i_r$ are the same, cancellation occurs in the sum $\sum_{\sigma \in S_r}$ and it becomes 0.
Therefore 
\begin{align*}
=& \sum_{\substack{i_1,\ldots,i_r=1\\
i_j \neq i_k}}^{N} \sum_{\sigma \in S_r}{\rm sgn}(\sigma) \overline{f_{i_1,1}}f_{i_1,\sigma(1)}\cdots \overline{f_{i_r,r}}f_{i_r,\sigma(r)}\\
=& \sum_{\substack{i_1,\ldots,i_r=1\\
i_j \neq i_k}}^{N} \sum_{\sigma \in S_r}{\rm sgn}(\sigma) f_{i_1,1}\overline{f_{i_1,\sigma(1)}}\cdots f_{i_r,r}\overline{f_{i_r,\sigma(r)}},
\end{align*}
because this sum is a real number.
On the other hand, we have
\begin{align*}
&\sum_{1\leq i_1<i_2< \cdots <i_{r}\leq N} \left|\det (s_{i_1} \, s_{i_2} \cdots s_{i_r} ) \right|^2\\
= & \sum_{i_1< \cdots <i_{r}} \left[\sum_{\sigma \in S_r} {\rm sgn} (\sigma) f_{i_1,\sigma(1)} \cdots f_{i_r,\sigma(r)} \right]\left[\sum_{\tau \in S_r} {\rm sgn} (\tau) \overline{f_{i_1,\tau(1)}} \cdots \overline{f_{i_r,\tau(r)}} \right]\\
= & \sum_{i_1< \cdots <i_{r}} \sum_{\sigma \in S_r} \sum_{\tau \in S_r} {\rm sgn}(\sigma \tau) f_{i_1,\sigma(1)}\overline{f_{i_1, \tau(1)}} \cdots f_{i_r,\sigma(r)}\overline{f_{i_r, \tau(r)}},
\end{align*}
We rearrange each term in this sum so that the sequence $\sigma(1),\ldots,\sigma(r)$ becomes $1,2,\ldots,r$. Then we obtain 
\begin{align*}
= & \sum_{i_1< \cdots <i_{r}} \sum_{\sigma \in S_r} \sum_{\tau \in S_r} {\rm sgn} (\sigma \tau) f_{i_{\sigma^{-1}(1)},1} \overline{f_{i_{\sigma^{-1}(1)},\tau(\sigma^{-1}(1))}}\cdots f_{i_{\sigma^{-1}(r)},r} \overline{f_{i_{\sigma^{-1}(r)},\tau(\sigma^{-1}(r))}}\\
= & \sum_{\substack{i_1,\ldots,i_r=1\\
i_j \neq i_k}}^{N} \sum_{\sigma \in S_r}{\rm sgn}(\tau \sigma^{-1}) f_{i_1,1}\overline{f_{i_1,\tau(\sigma^{-1}(1))}}\cdots f_{i_r,r}\overline{f_{i_r,\tau(\sigma^{-1}(r))}}.
\end{align*}
Here we use ${\rm sgn}(\sigma\tau) = {\rm sgn}(\tau \sigma^{-1})$.

Thus we have that
\[\det(h^{-1}) =  \sum_{1\leq j_1<j_2< \cdots <j_{r}\leq N} \left|\det (s_{j_1} \, s_{j_2} \cdots s_{j_r} ) \right|^2.\]

By changing basis, it is sufficient to prove in the case $s = (1,0,\cdots,0)$.
In this case, $|s|_h = h_{11}$ holds.
We have to calculate $h_{11}$, which can be written using (1,1)-cofactor of $h^{-1}$ and $\det (h^{-1})$.
We can calculate this cofactor similarly as above, because (1,1)-cofactor is the determinant of $(n-1) \times (n-1)$ submatrix.
This completes the proof.
\end{proof}

Now we continue the proof of Proposition \ref{coh}.
By the lemma, $|s|^2_h$ is locally integrable if and only if each term in the right hand side
$$\left|\det(s \, s_{i_1} \, s_{i_2} \cdots s_{i_{r-1}} )\right|^2  /  \sum \left|\det(s_{j_1} \, s_{j_2} \cdots s_{j_r} )\right|^2$$ is locally integrable for all $1 \leq i_1<i_2< \cdots < i_{r-1} \leq N$.
We define a multiplier ideal sheaf $\mathcal{J}$ by using a weight $\phi = \log \sum \left|\det
(s_{j_1} \, s_{j_2} \cdots s_{j_r} )
\right|^2$, i.e.\
\[\mathcal{J}(U) = \left\{f \in \mathcal{O}(U) ; \frac{|f|^2}{\sum_{1\leq j_1<j_2< \cdots <j_{r}\leq N} \left|\det
(s_{j_1} \, s_{j_2} \cdots s_{j_r} )
\right|^2} \in L^1_{loc}\right\}.\]
Then the condition $s \in E(h)$ is equivalent to the condition that 
$\det(s \, s_{i_1} \, s_{i_2} \cdots s_{i_{r-1}} ) \in \mathcal{J}$ for each $i_1, i_2, \ldots, i_{r-1}$.
We have that $\mathcal{J}$ is coherent, because multiplier ideals are coherent.
It follows that, for each $i_1, i_2, \ldots, i_{r-1}$, the sheaf of sections $s$ satisfying $\det(s \, s_{i_1} \, s_{i_2} \cdots s_{i_{r-1}} ) \in \mathcal{J}$ is coherent.
Since $E(h)$ is the intersection of all such sheaves, we have that $E(h)$ is coherent. Note that it is a finite intersection of coherent sheaves.

Now we consider the general case.
We denote by $h_{\rm Euc}$ the standard metric on $\mathbb{C}^N$.
Then locally we can write as
\[Ch_{\rm Euc} \leq h_0 \leq C' h_{\rm Euc}\]
for some $C,C'>0$.
Therefore, taking quotient, we have that
\[Ch_1 \leq h \leq C'h_1,\]
where $h_1$ denotes the quotient metric induced by $h_{\rm Euc}$. It follows that
\[C|s|_{h_1} \leq |s|_h \leq C' |s|_{h_1}\]
for any section $s\in E$.
This shows that $E(h_1)=E(h)$ and $E(h_1)$ is coherent by the preceding discussion.
\end{proof}

Here, for the next example, we will represent the curvature condition in Remark \ref{remNak} in the matrix form.
In the following, we regard $h$ as its representation matrix.
The Chern curvature of $h$ is written as $\Theta = \dbar(\overline{h}^{-1} \partial\overline{h}) = \sum \Theta_{ij} dz_i \wedge d\overline{z} _j$, where $\Theta_{ij}$ is a section of ${\rm End}(E)$. Then, the Hermitian form on $E \otimes T_X$ induced by $\Theta$ can be written as follows:
\[\Theta(s \otimes \xi, t \otimes \eta)=\{ \Theta s,t \}_h (\xi, \overline{\eta})= ({}^t s {}^t \Theta h t ) (\xi,\overline{\eta})\]
for $s,t \in E$ and $\xi,\eta \in T_X$.
If we take $e_i \otimes \partial/\partial z_j$ ($i =1,2,\ldots, r$, $j=1,2,\ldots, n$) for the frame of $E \otimes T_X$, the corresponding matrix representation of $\Theta(\cdot,\cdot)$ is as follows:
\[\Theta_{\rm Nak}=
\begin{pmatrix}
 {}^t\Theta_{1,1} h & {}^t\Theta_{1,2} h & \cdots & {}^t\Theta_{1,n} h \\
 {}^t\Theta_{2,1} h & {}^t\Theta_{2,2} h & \cdots & {}^t\Theta_{2,n} h \\
 \vdots & \vdots & \ddots & \vdots \\
 {}^t\Theta_{n,1} h & {}^t\Theta_{n,2} h & \cdots & {}^t\Theta_{n,n} h
\end{pmatrix}.
\]
Here, each ${}^t\Theta_{ij} h$ is an $r\times r$ matrix and then $\Theta_{\rm Nak}$ is an $nr \times nr$ matrix.
Then the condition that the curvature of $h$ is bounded from below in the sense of Nakano is equivalent to the condition that the following matrix is nonnegative:
\[\Theta_{\rm Nak} + C \cdot (\omega_{ij} h)_{i,j},\]
where we write $\omega = \sum \omega_{ij} idz_i \wedge d \overline{z}_j$.

\begin{example}\label{exampleapproximate}
Take $X=\mathbb{C}^2$. Let $E = X \times \mathbb{C}^2$ be a trivial rank two bundle. Let $(z,w)$ be the standard coordinate on $X$. We choose sections $s_1=(1,0), s_2=(z,w)$. Then the metric induced by $s_1, s_2$ on $E$ can be written as
\[h = 
\frac{1}{|w|^2}
\begin{pmatrix}
|w|^2 & -w\overline{z}\\
-z\overline{w} & |z|^2+1\\
\end{pmatrix}.
\]
Every entry of $h$ is smooth on $\mathbb{C}^2 \setminus \{w=0\}$.
The dual of $h$ is 
\[h_{\rm dual}= {}^t h^{-1} =
\begin{pmatrix}
 |z|^2+1 & z \overline{w}\\
 w \overline{z} & |w|^2\\
\end{pmatrix}.\]

We approximate $h$ in two ways, and calculate eigenvalues of $\Theta_{\rm Nak} + C \cdot (\omega_{ij} h)_{i,j}$. As the consequence, we will see both the approximations do not have bounded curvature below in the sense of Nakano. In particular, we will show that {\it the Nakano eigenvalue of the approximation of $h$ obtained by convolution is not bounded below}.

We consider the following two approximations of $h_{\rm dual}$:
\begin{align*}
h_{{\rm dual},\varepsilon}= &
\begin{pmatrix}
  |z|^2 + 1 & z \overline{w} \\
  w \overline{z} & |w|^2 +\varepsilon
\end{pmatrix}
= h_{\rm dual} + \varepsilon
\begin{pmatrix}
  0 &  0\\
  0 &  1\\
\end{pmatrix}
\\
h'_{{\rm dual},\varepsilon} = &
\begin{pmatrix}
  |z|^2 + 1 + \varepsilon & z \overline{w} \\
  w \overline{z} & |w|^2 +\varepsilon
\end{pmatrix}
= h_{\rm dual} + \varepsilon
\begin{pmatrix}
  1 &  0\\
  0 &  1\\
\end{pmatrix}
\end{align*}

Let $h_\varepsilon$ and $h'_\varepsilon$ be the dual metrics of $h_{{\rm dual},\varepsilon}$ and $h'_{{\rm dual},\varepsilon}$, respectively.

Note that $h'_{{\rm dual},\varepsilon}$ is obtained by convolution of $h_{\rm dual}$ by an appropriate smooth kernel function. 
Indeed, let $\chi$ be a smooth function with compact support on $\mathbb{C}^2$ such that $\chi$ depends only on $|(z,w)|=|z|^2+|w|^2$ and $\int_{\mathbb{C}^2} \chi d\lambda = 1$. 
Let $f(z,w):= |z|^2$.
We will show that $\chi * f = f + \varepsilon_\chi$, where $\varepsilon_\chi>0$ is a constant.
Let $x,p \in \mathbb{C}^2$.
Then we have
\begin{align*}
\chi * f (x) &= \int_{\mathbb{C}^2} \chi(p) f(x-p) d\lambda(p)\\
&= \frac{1}{2}\int_{\mathbb{C}^2}\chi(p) f(x-p) d\lambda(p) + \frac{1}{2}\int_{\mathbb{C}^2}\chi(p) f(x-p) d\lambda(p)\\
&= \frac{1}{2}\int_{\mathbb{C}^2}\chi(p) f(x-p) d\lambda(p) + \frac{1}{2}\int_{\mathbb{C}^2}\chi(-p) f(x+p) d\lambda(p)\\
&= \int_{\mathbb{C}^2}\chi(p) \frac{1}{2}(f(x-p)+f(x+p) d\lambda(p),
\end{align*}
because $\chi$ is symmetric under $p\mapsto -p$.
We have that $f(x-p)+f(x+p) = 2|x_1|^2 + 2|p_1|^2$, where $x=(x_1,x_2),p=(p_1,p_2)$.
Thus,
\begin{align*}
\chi*f(x) &= \int_{\mathbb{C}^2}\chi(p) (|x_1|^2+|p_1|^2) d\lambda(p)\\
&= |x_1|^2 + \int_{\mathbb{C}^2}\chi(p)|p_1|^2 d\lambda(p)\\
&= f(x) + \varepsilon_\chi.
\end{align*}
Similarly we have $\chi * |w|^2 = |w|^2 + \varepsilon_\chi$ with the same constant $\varepsilon_\chi$ (because $\chi$ is symmetric under $(z,w)\mapsto(w,z)$).

For a function $g = z\overline{w}$, we have $\chi * g = g$ by similar calculation using an equation $g(z-p_1,w-p_2)+g(z+p_1,w-p_2)+g(z-p_1,w+p_2)+g(z+p_1,w+p_2) = 4g(z,w)$.

We denote by $\Theta_{\rm Nak , \varepsilon}$ and $\Theta'_{\rm Nak , \varepsilon}$ the corresponding matrix representation of the Hermitian form on $E \otimes T_X $ induced by $h_\varepsilon$ and  $h'_\varepsilon$, respectively. By calculating this matrix, we have that
$\Theta_{\rm Nak, \varepsilon} =  -\displaystyle\frac{\varepsilon}{(\varepsilon|z|^2+|w|^2+\varepsilon)^3} M$, where $M$ is a matrix
$$
\begin{pmatrix}-{\left( |w|^2+\varepsilon\right) }^{2} & w\,\left( |w|^2+\varepsilon\right) \,\overline{z} & w\,\left( |w|^2+\varepsilon\right) \,\overline{z} & -{w}^{2}\,{\overline{z}}^{2}\cr \overline{w}\,\left( |w|^2+\varepsilon\right) \,z & -|w|^2\,|z|^2 & -\left( |w|^2+\varepsilon\right) \,\left( |z|^2+1\right)  & w\,\overline{z}\,\left( |z|^2+1\right) \cr \overline{w}\,\left( |w|^2+\varepsilon\right) \,z & -\left( |w|^2+\varepsilon\right) \,\left( |z|^2+1\right)  & -|w|^2\,|z|^2 & w\,\overline{z}\,\left( |z|^2+1\right) \cr -{\overline{w}}^{2}\,{z}^{2} & \overline{w}\,z\,\left( |z|^2+1\right)  & \overline{w}\,z\,\left( |z|^2+1\right)  & -{\left( |z|^2+1\right) }^{2}\end{pmatrix},
$$
and 
$\Theta'_{\rm Nak, \varepsilon} =-\displaystyle\frac{\varepsilon(\varepsilon+1)}{(\varepsilon|z|^2+\varepsilon|w|^2+|w|^2+\varepsilon^2+\varepsilon)^3} M'$, where $M'$ is a matrix $$
\begin{pmatrix}-{\left( |w|^2+\varepsilon\right) }^{2} & w\,\left( |w|^2+\varepsilon\right) \,\overline{z} & w\,\left( |w|^2+\varepsilon\right) \,\overline{z} & -{w}^{2}\,{\overline{z}}^{2}\cr \overline{w}\,\left( |w|^2+\varepsilon\right) \,z & -|w|^2\,|z|^2 & -\left( |w|^2+\varepsilon\right) \,\left( |z|^2+\varepsilon+1\right)  & w\,\overline{z}\,\left( |z|^2+\varepsilon+1\right) \cr \overline{w}\,\left( |w|^2+\varepsilon\right) \,z & -\left( |w|^2+\varepsilon\right) \,\left( |z|^2+\varepsilon+1\right)  & -|w|^2\,|z|^2 & w\,\overline{z}\,\left( |z|^2+\varepsilon+1\right) \cr -{\overline{w}}^{2}\,{z}^{2} & \overline{w}\,z\,\left( |z|^2+\varepsilon+1\right)  & \overline{w}\,z\,\left( |z|^2+\varepsilon+1\right)  & -{\left( |z|^2+\varepsilon+1\right) }^{2}\end{pmatrix}.
$$

We claim that, for every constant $C>0$, there is $\varepsilon > 0$ which does not satisfy $\Theta_{\rm Nak,\varepsilon} \geq -C \omega \otimes {\rm Id}_E$ .
We take $\omega = idz \wedge d\overline{z} + idw \wedge d\overline{w}$. Then the matrix representations of $\omega \otimes {\rm Id}_E$ as a Hermitian form on $E \otimes T_X$ are
\[
\begin{pmatrix}
h_\varepsilon & 0\\
0 & h_\varepsilon\\
\end{pmatrix}
, {\rm or}
\begin{pmatrix}
h'_\varepsilon & 0\\
0 & h'_\varepsilon\\
\end{pmatrix}
{\rm , respectively}.
\]
We will show that for every fixed $C>0$, $\Theta_{\rm Nak,\varepsilon} + C \omega \otimes {\rm Id}_E$ (resp.\ $\Theta'_{\rm Nak,\varepsilon} + C \omega \otimes {\rm Id}_E$) has a negative eigenvalue for sufficiently small $\varepsilon$.
By direct computation, one of the eigenvalues of $\Theta_{{\rm Nak},\varepsilon} + C \omega \otimes {\rm Id}_E$ (resp.\ $\Theta'_{\rm Nak,\varepsilon} + C \omega \otimes {\rm Id}_E$) at $(z,w)=(0,0)$ is as follows:
\begin{align*}
\frac{(\varepsilon+1)C- \sqrt[]{(1-\varepsilon)^2 C^2 + 4}}{2\varepsilon} &\,\,\,\text{ for}\,\, \Theta_{{\rm Nak}, \varepsilon},\\
\frac{(2\varepsilon+1)C - \sqrt[]{C^2+4}}{2\varepsilon^2 + 2 \varepsilon} &\,\,\,\text{ for}\,\, \Theta'_{{\rm Nak}, \varepsilon}.
\end{align*}
For a fixed $C>0$, these eigenvalues go to $-\infty$ as $\varepsilon \to 0$.
Therefore, the Chern curvature of the approximations of $h_{\varepsilon}, h'_{\varepsilon}$ are not bounded below in the sense of Nakano.
\end{example}

\section{Existence of negatively curved singular Hermitian metrics on bundles given by extensions }\label{ClassifyGrifNegSingMet}
We consider an extension of vector bundles
$0 \to L \to E \to E' \to 0$ 
with a line bundle $L$.
We will show that the existence of a negatively curved singular Hermitian metric on $E$ with some conditions on $L$ implies splitting of this sequence. Using this, we can determine all negatively curved singular Hermitian metrics on certain vector bundles.

\begin{theorem}\label{determine}
Let $X$ be a compact complex manifold, $L$ a holomorphic line bundle on $X$, and let $E, E'$ be holomorphic vector bundles on $X$.
Assume that there is an exact sequence
\[0 \longrightarrow L \overset{i}\longrightarrow E \overset{p}\longrightarrow E' \longrightarrow 0. \]
Suppose that there are a holomorphic section $f \in H^0(X,L^*)$ and a negatively curved singular Hermitian metric $h$ on $E$ with $|i(s)|^2_h = | (f,s) |^2$ for each $s \in L$, where $( \cdot, \cdot )$ is a natural pairing on $L_x^* \times L_x$.
Then, the exact sequence above splits.
\end{theorem}

We begin the proof of Theorem \ref{determine} with local consideration.

\begin{proposition}\label{locrep}
Let $U$ be a (small) ball in $\mathbb{C}^n$ and let $h$ be a singular Hermitian metric on $U \times \mathbb{C}^r$.
We assume the representation matrix of $h$ has the form
\[h = 
\begin{pmatrix}
      |f|^2 & B \\
      \overline{B} & C  \\
\end{pmatrix},
\]
where $f$ is a holomorphic function and $B, C$ are measurable function valued $1 \times (r-1)$ and $(r-1) \times (r-1)$ matrix on $U$ ($r \geq 2$).
Assume that $h$ is negatively curved.
Then, there are holomorphic functions $g_i \in \mathcal{O}_U$ with $B_i=f \overline{g_i}$.
\end{proposition}

To prove Proposition \ref{locrep}, we use the following lemma.

\begin{lemma}\label{pshlim}
Let $\phi,\psi$ be locally integrable functions on $U \subset \mathbb{C}^n$.
Assume that $t\phi + \psi$ is plurisubharmonic for every $t>0$.
Then, $\psi$ equals to a plurisubharmonic function almost everywhere.
\end{lemma}

\begin{proof}
When $t \to 0$, $t\phi + \psi$ converges to $\psi$ as currents, so we have $i\ddbar (t\phi + \psi) \rightarrow i\ddbar\psi$ as currents. By assumption $i\ddbar (t\phi + \psi)$ is a positive current for any $t>0$. Therefore the limit $i\ddbar \psi$ is also positive, so we have $\psi$ equals to a plurisubharmonic function almost everywhere.
\end{proof}

\begin{proof}[\indent\sc Proof of Proposition \ref{locrep}]
We can assume $r=2$ and we denote $B_1 = B, C_{1,1} = C$. 
First, we assume $f$ is a constant function $f \equiv 1$ and prove $B = \overline{g}$ for some holomorphic function $g$.

Let $u\in \mathbb{C}$ be a constant. Then,
\[
\begin{pmatrix}
      u & 1  \\
\end{pmatrix}
\begin{pmatrix}
      1 & B \\
      \overline{B} & C  \\
\end{pmatrix}
\begin{pmatrix}
      \overline{u}\\
      1\\
\end{pmatrix}
= |u|^2 + \overline{u}\overline{B} + uB + C
\]
is plurisubharmonic. Since $|u|^2$ is constant, $\overline{u}\overline{B} + uB + C$ is also plurisubharmonic.
Taking $u \in \mathbb{R}$ and $u \in i\mathbb{R}$, we have that $2t\,{\rm Re}\, B + C$ and $2t\,{\rm Im}\,B + C$ are plurisubharmonic for every $t \in \mathbb{R}$.
Then, $2{\rm Re}B + \frac{1}{t} C$ and $2{\rm Im}B + \frac{1}{t} C$ are plurisubharmonic for $t \neq 0$ and Lemma \ref{pshlim} shows that $B$ is a (complex-valued) pluriharmonic function.
It follows that $B$ can locally be written as a sum of a holomorphic function and an antiholomorphic function, namely $B = g_1 + \overline{g_2}$.

Since $g_1$ is holomorphic, we have that for any $u\in \mathbb{C}$
\[
\begin{pmatrix}
      u & g_1 \\
\end{pmatrix}
\begin{pmatrix}
      1 & g_1 + \overline{g_2} \\
      \overline{g_1} + g_2 & C  \\
\end{pmatrix}
\begin{pmatrix}
      \overline{u}\\
      \overline{g_1}\\
\end{pmatrix}
= |u|^2 + \overline{u}|g_1|^2+\overline{u}g_1g_2+ u|g_1|^2+u\overline{g_1}\overline{g_2} + |g_1|^2 C
\]
is plurisubharmonic. The term $|u|^2$ is constant, and $\overline{u}g_1g_2, u\overline{g_1}\overline{g_2}$ are pluriharmonic since they are holomorphic and antiholomorphic respectively. Thus we have $\overline{u}|g_1|^2 +  u|g_1|^2 + |g_1|^2 C$ is plurisubharmonic for every $u \in \mathbb{C}$. Lemma \ref{pshlim} shows that $|g_1|^2$ is pluriharmonic. Therefore we have
\[i\ddbar |g_1|^2 = i\partial g_1 \wedge \dbar \overline{g_1} =0,\]
which implies that $\partial g_1=0$. Thus $g_1$ is constant and $B=g_1+\overline{g_2}$ is antiholomorphic, hence the proposition holds when $f=1$.

For the general case, we can similarly show that $B/f$ is antiholomorphic on $\{f \neq 0\}$ using sections $(u/f,1),(u/f,g_1)$ instead of $(u,1),(u,g_1)$.
Since $h$ is nonnegative, we have $\det h = |f|^2 C - |B|^2 \geq 0$,
thus $|B/f|^2 \leq C$.
Since $C=|(0,1)|^2_h$ is plurisubharmonic, $C$ is locally bounded from above. Therefore, $|B/f|$ is locally bounded. Riemann's extension theorem implies that $B/f$ is an antiholomorphic function on $U$.
\end{proof}

\begin{proof}[\indent\sc Proof of Theorem \ref{determine}]
Let ${\rm rank\,} E = r$.
Let $\{U_\alpha\}$ be a covering of $X$ by sufficiently small open sets.
We assume that given bundles are trivial on each $U_\alpha$.
We denote a trivializing section of $L$ on $U_\alpha$ by $s_\alpha$ and a holomorphic frame of $E'$ on $U_\alpha$ by $e'_{\alpha,2}, e'_{\alpha,3}, \ldots, e'_{\alpha,r}$.
We take a holomorphic frame $e_{\alpha,1}, e_{\alpha,2}, \ldots, e_{\alpha,r}$ of $E$ satisfying
\begin{align*}
i(s_\alpha) &= e_{\alpha,1},\\
p(e_{\alpha, i}) &= e'_{\alpha, i}. \ \  (i=2,\ldots,r)
\end{align*}
The transition function of $L$ and $E'$ are denoted by $g_{\alpha\beta}$ and $g'_{\alpha\beta,i,j}$ as follows:
\begin{align*}
s_{\alpha} &= g_{\alpha\beta}s_{\beta},\\
e'_{\alpha,i} &= \sum_{2\leq j \leq r} g'_{\alpha\beta,i,j} e'_{\beta,j}. \ \  (i=2,\ldots,r)
\end{align*}
Then we have that $p(e_{\alpha, i} - \sum_j g'_{\alpha\beta,i,j} e_{\beta,j})=0$, thus there exist holomorphic functions $h_{\alpha\beta,i}$ on $U_{\alpha} \cap U_{\beta}$, $i=2,\ldots,r$, satisfying
\[e_{\alpha, i} - \sum_{2\leq j \leq r} g'_{\alpha\beta,i,j} e_{\beta,j} = h_{\alpha\beta,i} e_{\beta,1}.\]
In this notation, the transition function of $E$ can be written as
\begin{align*}
e_{\alpha,1} &= g_{\alpha\beta}e_{\beta,1},\\
e_{\alpha,i} &= h_{\alpha\beta,i} e_{\beta,1} + \sum_{2\leq j \leq r} g'_{\alpha\beta,i,j}e_{\beta,j}. \ \ (i=2,\ldots,r)
\end{align*}
Let $s$ be a section of $E$.
When we write $s = a_1 e_{\alpha,1}+ \cdots + a_r e_{\alpha,r} = b_1 e_{\beta,1}+ \cdots + b_r e_{\beta,r}$, the transition function is as follows:
\[
\begin{pmatrix}
b_1\\
b_2\\
\vdots\\
b_r\\
\end{pmatrix}
=
\begin{pmatrix}
g_{\alpha\beta} & h_{\alpha\beta,2} & \cdots & h_{\alpha\beta,r}\\
0            & g'_{\alpha\beta,2,2} & \cdots & g'_{\alpha\beta,r,2}\\
\vdots          & \vdots            & \ddots & \vdots \\
0            & g'_{\alpha\beta,2,r} & \cdots & g'_{\alpha\beta,r,r}
\end{pmatrix}
\begin{pmatrix}
a_1\\
a_2\\
\vdots\\
a_r\\
\end{pmatrix}.
\]
We will denote this matrix by $G_{\alpha\beta}$.

We denote the local matrix representation of $h$ by
\[h_\alpha = 
\begin{pmatrix}
A_\alpha & B_\alpha\\
{}^t\overline{B}_\alpha & C_\alpha
\end{pmatrix},
\]
where $A_\alpha$ is a scalar, $B_\alpha = (B_{\alpha,2},\ldots,B_{\alpha,r})$ is a $1\times (r-1)$ matrix, and $C_\alpha$ is a $(r-1)\times(r-1)$ matrix.
We write the given section $f \in H^0(X,L^*)$ as $f = f_\alpha s_\alpha^{-1}$, where $f_\alpha$ is a holomorphic function on $U_\alpha$ and $s_\alpha^{-1}$ is a dual of $s_\alpha$.
By $s_\alpha = g_{\alpha\beta}s_\beta$, we have $f_\alpha = g_{\alpha\beta}f_\beta$.
By assumption $|s|^2_h = |(f,s)|^2_h$ for $s \in L$, we have $A_\alpha = |e_{\alpha,1}|^2_h = |(f,e_{\alpha,1})|^2 = |f_\alpha|^2$.
The transition function for $h_\alpha$ is as follows:
\[h_\alpha = {}^t G_{\alpha\beta} h_\beta \overline{G_{\alpha\beta}}.\]
By calculating the first row, we have that
\begin{align*}
B_{\alpha,i} &= g_{\alpha\beta}A_{\beta}\overline{h_{\alpha\beta,i}}+\sum_{2 \leq j \leq r} g_{\alpha\beta}B_{\beta,j}\overline{g'_{\alpha\beta,i,j}}\\
 &= |f_\beta|^2 g_{\alpha\beta}\overline{h_{\alpha\beta,i}}+\sum_{2 \leq j \leq r} g_{\alpha\beta}B_{\beta,j}\overline{g'_{\alpha\beta,i,j}}.
\end{align*}
By dividing both sides by $f_\alpha = g_{\alpha\beta}f_\beta$, we have
\begin{align*}
B_{\alpha,i}/f_{\alpha} &= \frac{|f_\beta|^2 g_{\alpha\beta}\overline{h_{\alpha\beta,i}}}{g_\alpha\beta f_\beta}+ \frac{\sum_{2 \leq j \leq r} g_{\alpha\beta}B_{\beta,j}\overline{g'_{\alpha\beta,i,j}}}{g_{\alpha\beta}f_{\beta}}\\
&= \overline{f_\beta h_{\alpha\beta,i}} + \sum_j \overline{g'_{\alpha\beta,i,j}} B_{\beta,j}/f_{\beta}.
\end{align*}
Let $\gamma_{\alpha,i} = \overline{B_{\alpha,i}/f_{\alpha}}$, then $\gamma_{\alpha,i}$ satisfies
\[\gamma_{\alpha,i} = f_\beta h_{\alpha\beta,i} + \sum_j g'_{\alpha\beta,i,j} \gamma_{\beta,j}.\]
By the Proposition \ref{locrep}, $\gamma_{\alpha,i}$ is a holomorphic function on $U_\alpha$.

Let $\xi'_{\alpha,2},\ldots,\xi'_{\alpha,r}$ be the dual frame in $(E')^*$ of $e'_{\alpha,2},\ldots,e'_{\alpha,r}$.
We consider a \v{C}ech 0-cochain $\left(\sum_{2 \leq i \leq r} \gamma_{\alpha,i} \xi'_{\alpha,i} \otimes f^{-1}, U_\alpha\right).$
Then, the differential of this cochain is
\[\sum_{2 \leq i \leq r} \gamma_{\alpha,i} \xi'_{\alpha,i} \otimes f^{-1} - \sum_i \gamma_{\beta,i} \xi'_{\beta,i} \otimes f^{-1} = \sum_i \frac{h_{\alpha\beta,i}}{g_{\alpha\beta}} \xi'_{\alpha,i}s_{\alpha}.
\]
Next we consider the extension class in $H^1(X, L\otimes (E')^*)$ of given exact sequence.
It is known that the extension class is the image of ${\rm Id}_{E'} \in H^0(X,E' \otimes (E')^*)$ by the connecting homomorphism $H^0(X,E' \otimes (E')^*) \to H^1(X, L \otimes (E')^*)$ induced by the following short exact sequence:
\[0 \to L \otimes (E')^* \to E \otimes (E')^* \to E' \otimes (E')^* \to 0.\]
We can calculate this class using following diagram:
\[\xymatrix{
	C^0(U_\alpha, L \times (E')^*)\ar@{->}[r]\ar@{->}[d]& C^0(U_\alpha, E \times (E)^*)\ar@{->}[r]\ar@{->}[d]& C^0(U_\alpha, E' \times (E)^*)\ar@{->}[d]\\
	C^1(U_\alpha, L \times (E')^*)\ar@{->}[r]& C^1(U_\alpha, E \times (E)^*)\ar@{->}[r]& C^1(U_\alpha, E' \times (E)^*),\\
}
\]
where $C^i(U_\alpha,F)$ denotes the space of \v{C}ech $i$-cochains. Then, we have the following:
\[\xymatrix{
& \sum_{j=2}^r(e_{\alpha,j} \otimes \xi'_{\alpha,j})\ar@{|->}[r] \ar@{|->}[d]
& \sum_{j=2}^r(e'_{\alpha,j} \otimes \xi'_{\alpha,j})\\
\sum_{j=2}^r\displaystyle\frac{h_{\alpha\beta,j}}{g_{\alpha\beta}}\xi'_{\alpha,j}s_\alpha \ar@{|->}[r]& \sum_{j=2}^r(e_{\alpha,j} \otimes \xi'_{\alpha,j} - e_{\beta,j} \otimes \xi'_{\beta,j}). &\\
}
\]
Calculating the map in the second row, we can show that the extension class is identical to the differential of the 0-cochain described above.
Thus the extension class is represented by a exact 1-cocycle, which implies that given extension is trivial.
\end{proof}

\begin{example}\label{EllipticCurve}
Using Theorem \ref{determine}, we can determine all negatively curved singular Hermitian metrics on a nontrivial rank two vector bundle on an elliptic curve, which appeared in \cite{DPS}, Example 1.7.
Let $C$ be an elliptic curve. We define a vector bundle $E$ on $C$ by the nontrivial exact sequence
\[0 \longrightarrow \mathcal{O}_1 \longrightarrow E \longrightarrow \mathcal{O}_2 \longrightarrow 0,\]
where $\mathcal{O}_1 = \mathcal{O}_2 =\mathcal{O}_C$.
Note that $\dim H^1(C,\mathcal{O}_C) = 1$, thus $E$ is uniquely determined up to isomorphism.

There is a more concrete description of $E$ in \cite{DPS}.
We can obtain $E$ as the quotient $E= \mathbb{C} \times \mathbb{C}^2 / \Gamma$, where $\Gamma=\mathbb{Z} + \tau \mathbb{Z}$ is a lattice for the elliptic curve $E$. 
Here, an action of $1, \tau \in \Gamma$ to the space $\mathbb{C} \times \mathbb{C}^2$ is described by
$(x,z_1,z_2) \mapsto (x+1,z_1,z_2)$ and $(x,z_1,z_2) \mapsto (x+\tau,z_1+z_2,z_2)$.

Let $h$ be a negatively curved singular Hermitian metric on $E$.
Then the restriction $h|_{\mathcal{O}_1}$ is also negatively curved.
A negatively curved metric on the trivial line bundle corresponds to a subharmonic function on $C$ via $h \mapsto |1|^2_h$.
Since any subharmonic function on a compact Riemann surface is constant, $h|_{\mathcal{O}_1}$ is also constant and we can write this constant by $h|_{\mathcal{O}_1} = C_0$.
If $C_0 \neq 0$, the assumption of Theorem \ref{determine} is satisfied and it follows that given exact sequence splits.
It contradicts the definition of $E$. Therefore, we have $h|_{\mathcal{O}_1} = 0$.
Moreover, we can show that $h$ has the form $p^*(h')$, where $h'$ is a metric on $\mathcal{O}_2$ which is negatively curved (Lemma \ref{MetricP}). This curvature condition implies $h'$ is constant.

In this example, we showed that $\det h \equiv 0$ for every singular Hermitian metrics $h$ on $E$. This is why we admit singular metrics with $\det h=0$ everywhere.
\end{example}

\begin{lemma}\label{MetricP}
Let $L, L'$ be line bundles, $0 \to L \to E \overset{p}{\to} L' \to 0$ be an exact sequence, and $h$ be a singular Hermitian metric on $E$.
Assume that $h|_L \equiv 0$.
Then, there exists a singular Hermitian metric $h'$ on $L'$ with $h'(p(s)) = h(s)$ for $s \in E$.
Moreover, if $h$ is negatively curved, so is $h'$.
\end{lemma}

\begin{proof}
Let $s_0, s'_0$ be trivializing sections of $L, L'$ respectively.
We take a local holomorphic frame $e_1, e_2$ of $E$ with
\[e_1 = i(s_0), p(e_2) = s'_0,\]
where $i \colon L \to E$.
Then $h(e_1,e_1) = 0$ by assumption.
It follows that $h(e_1,e_2) = 0$ by the Cauchy-Schwarz inequality.

For each $s' \in E'$, take $s$ which satisfies $p(s) = s'$ and define $h'(s')$ by $h(s)$.
Since we have $h(e_1,e_2) = 0$ in the local frame, this definition is independent of the choice of $s$.
Then we have for $s \in E$
\[(p^* h')(s) = h'(p(s)) = h(s),\]
here the second equality is by the definition of $h'$.

Assume that $h$ is negatively curved.
For holomorphic section $s' \in \mathcal{O}(E')$, we can find a holomorphic section $s \in \mathcal{O}(E)$ with $p(s) = s'$.
Therefore $h'(s') = h(s)$ is a plurisubharmonic function. This proves that $h'$ is also negatively curved.
\end{proof}


\providecommand{\bysame}{\leavevmode\hbox to3em{\hrulefill}\thinspace}
\providecommand{\MR}{\relax\ifhmode\unskip\space\fi MR }
\providecommand{\MRhref}[2]{%
  \href{http://www.ams.org/mathscinet-getitem?mr=#1}{#2}
}
\providecommand{\href}[2]{#2}


\begin{thebibliography}{10}

\bibitem{Ber}
B.~Berndtsson, \emph{Curvature of vector bundles associated to holomorphic
  fibrations}, Ann. of Math. (2) \textbf{169} (2009), no.~2, 531--560.

\bibitem{BP}
B.~Berndtsson and M.~P{\u{a}}un, \emph{Bergman kernels and the
  pseudoeffectivity of relative canonical bundles}, Duke Math. J. \textbf{145}
  (2008), no.~2, 341--378.

\bibitem{BK}
Z.~B{\l}ocki and S.~Ko{\l}odziej, \emph{On regularization of plurisubharmonic
  functions on manifolds}, Proc. Amer. Math. Soc. \textbf{135} (2007), no.~7,
  2089--2093. 
  
\bibitem{CLS}
D.~A. Cox, J.~B. Little, and H.~K. Schenck, \emph{Toric varieties}, Graduate
  Studies in Mathematics, vol. 124, American Mathematical Society, Providence,
  RI, 2011.

\bibitem{deC}
M.~A.~A. de~Cataldo, \emph{Singular {H}ermitian metrics on vector bundles}, J.
  Reine Angew. Math. \textbf{502} (1998), 93--122.

\bibitem{DemCG}
J.-P. Demailly, \emph{Complex analytic and differential geometry}, \url{http://www-fourier.ujf-grenoble.fr/~demailly/manuscripts/agbook.pdf}.

\bibitem{DemAM}
J.-P. Demailly, \emph{Analytic methods in algebraic geometry}, Surveys of Modern
  Mathematics, vol.~1, International Press, Somerville, MA; Higher Education
  Press, Beijing, 2012.

\bibitem{DPS}
J.-P. Demailly, T.~Peternell, and M.~Schneider, \emph{Compact complex manifolds
  with numerically effective tangent bundles}, J. Algebraic Geom. \textbf{3}
  (1994), no.~2, 295--345.

\bibitem{Lem}
L.~{Lempert}, \emph{Modules of square integrable holomorphic germs}, \url{arXiv:1404.0407}.

\bibitem{LSY}
K.~Liu, X.~Sun, and X.~Yang, \emph{Positivity and vanishing theorems for ample
  vector bundles}, J. Algebraic Geom. \textbf{22} (2013), no.~2, 303--331.

\bibitem{Nad}
A. M. ~Nadel, \emph{Multiplier ideal sheaves and {K}\"ahler-{E}instein metrics of positive scalar curvature}, Proc. Nat. Acad. Sci. U.S.A., \textbf{86} (1989), 7299?730 and Ann. of Math. \textbf{132} (1990), no.~3, 549--596.

\bibitem{PT}
M.~P{\u{a}}un and S.~Takayama, \emph{Positivity of twisted relative pluricanonical bundles and their direct images}, \url{arXiv:1409.5504}.

\bibitem{Rau}
H.~{Raufi}, \emph{Singular hermitian metrics on holomorphic vector bundles}, Ark. Mat. \textbf{53}, (2015), no.~2, 359--382.

\bibitem{Y}
Q.~{Yang}, \emph{{$L^2$}-extension theorems for jet sections of nef holomorphic
  vector bundles on compact kahler manifolds and rational homogeneous manifolds, {I}}, \url{arXiv:1412.7869v1}.

\end{thebibliography}
\end{document}